\documentclass{amsart}

\usepackage[T1]{fontenc}
\usepackage[utf8]{inputenc}

\usepackage{amssymb,amsmath,amsfonts,amsthm,mathrsfs,mathtools,amscd,tikz-cd,dirtytalk,comment}

\newtheorem{theorem}{Theorem}[section]
\newtheorem{lemma}[theorem]{Lemma}
\newtheorem{proposition}[theorem]{Proposition}
\newtheorem{corollary}[theorem]{Corollary}
\newtheorem{conjecture}[theorem]{Conjecture}

\theoremstyle{definition}
\newtheorem{definition}[theorem]{Definition}
\newtheorem{example}[theorem]{Example}

\theoremstyle{remark}
\newtheorem{remark}[theorem]{Remark}

\numberwithin{equation}{section}

\DeclareMathOperator{\Aut}{Aut}

\DeclareMathOperator{\GL}{GL}

\DeclareMathOperator{\id}{id}

\DeclareMathOperator{\Inj}{Inj}

\DeclareMathOperator{\opTriv}{opTriv}
\DeclareMathOperator{\Ret}{Ret}
\DeclareMathOperator{\Soc}{Soc}
\DeclareMathOperator{\Perm}{Perm}

\DeclareMathOperator{\Triv}{Triv}

\newcommand{\N}{\mathbb{N}}
\newcommand{\Z}{\mathbb{Z}}

\newcommand{\G}{\mathcal{G}}
\newcommand{\hsigma}{\hat{\sigma}}
\newcommand{\op}{\mathrm{op}}
\newcommand{\sub}{\subseteq}
\newcommand{\Sym}{\mathbb{S}}

\newcommand{\lexp}[2]{\null^{#2} \mkern-2mu #1}
\newcommand{\lexpp}[2]{\null^{#2} \mkern-2mu (#1)}

\usepackage[hidelinks]{hyperref}

\begin{document}

\title[On bi-skew braces and brace blocks]{On bi-skew braces and brace blocks}

\author{L.~Stefanello}

\address{Department of Mathematics,
          Universit\`a di Pisa,
          Largo Bruno Pontecorvo 5, 56127 Pisa, Italy}
        
\email{lorenzo.stefanello@phd.unipi.it}

\urladdr{https://people.dm.unipi.it/stefanello/}
\thanks{The first author was a member of GNSAGA (INdAM)}

\author{S.~Trappeniers}

\address{Department of Mathematics and Data Science, 
 Vrije Universiteit Brussel, 
 Pleinlaan 2, 
 1050 Brussels, Belgium} 

\email{senne.trappeniers@vub.be} 
\thanks{The second author was supported by Fonds voor Wetenschappelijk Onderzoek (Flanders), via an FWO Aspirant-fellowship, grant 1160522N}



\begin{abstract}
L.~N.~Childs defined a bi-skew brace to be a skew brace such that if we swap the role of the two operations, then we find again a skew brace.

In this paper, we give a systematic analysis of bi-skew braces. We study nilpotency and solubility, and connections between bi-skew braces and set-theoretic solutions of the Yang--Baxter equation. Further, we deal with Byott's conjecture in the case of bi-skew braces, and we use bi-skew braces as a tool to solve a classification problem proposed by L.~Vendramin.

In the final part, we investigate brace blocks, defined by A.~Koch to be families of group operations on a given set such that any two of them yield a bi-skew brace. We provide a characterisation of brace blocks, illustrate how all known constructions in literature follow in a natural way from our characterisation, and give several new examples.
\end{abstract}

\maketitle

\section{Introduction}

After the definition of skew braces in~\cite{GV17}, built on the pioneering work of~\cite{Rum07a}, a considerable part of literature has been devoted to the study of the main properties of these objects and their relations with other topics, such as Jacobson radical rings, regular subgroups of the holomorph, the Yang--Baxter equation, and Hopf--Galois structures~\cite{Bac18,Ced18,SV18,BCJO19}. In particular, these relations motivated even more the study of the involved topics, and allowed various problems to be translated in different settings. For example, one deep problem regarding the properties of the underlying groups of a skew brace was initially formulated by N.~P.~Byott as a statement on regular subgroups of the holomorph of a soluble group~\cite{Byo15}, and it is now known as Byott's conjecture.

 In~\cite{Chi19}, L.~N.~Childs defined a bi-skew brace to be a skew brace $(A,\cdot,\circ)$ such that also $(A,\circ,\cdot)$ is a skew brace. His main focus was the relation with Hopf--Galois theory; a bi-skew brace $(A,\cdot,\circ)$ with $A$ finite provides not only a Hopf--Galois structure of type $(A,\cdot)$ on every Galois extension with Galois group $(A,\circ)$, but also a Hopf--Galois structure  of type $(A,\circ)$ on every Galois extension with Galois group $(A,\cdot)$; see~\cite{Chi00} for a general treatment of Hopf--Galois theory, and the appendix of~\cite{SV18},~\cite[Chapter 2]{CGKKKTU21}, and~\cite{ST22b-p} for the relation between skew braces and Hopf--Galois structures.

A.~Caranti studied different characterisations of bi-skew braces and gave various constructions in~\cite{Car20}. These characterisations were formulated both using gamma functions and from the point of view of regular subgroups of the holomorph.
Bi-skew braces were further studied by A.~Koch in~\cite{Koc21}, where a construction for bi-skew braces is given starting from group endomorphisms with abelian image. An iterative version of this construction was then obtained in~\cite{Koc22}, where also the notion of a brace block is given. This is a family of group operations on a fixed set such that any two operations form a skew brace. In particular, every bi-skew brace yields a brace block with two operations. The constructions by Koch and a construction of Caranti~\cite{Car18} were subsequently generalised by Caranti and the first author in~\cite{CS21,CS22a}. In the recent manuscript~\cite{BNY22-p}, an iterative construction is given to obtain a brace block from a given bi-skew brace satisfying a certain property. A common factor in these works is the study of the questions, "How can we better understand bi-skew braces or brace blocks? What are new, effective ways to construct them?" The same idea lives on throughout this paper. In particular, a better understanding of bi-skew braces and brace blocks leads to new ways of constructing them and explains known constructions and their conditions in a natural way.

The paper is organised as follows. 

In section~\ref{sec: prelim} we recall the preliminaries that are used throughout the paper.

In section~\ref{sec: structural result} we state several structural results of bi-skew braces. We relate structural properties of a bi-skew brace to those of its associated skew brace with swapped operations and also to properties of a suitable group associated with the skew brace. Moreover, we give an affirmative answer to Byott's conjecture in the case of bi-skew braces, employing an approach that also works for several other classes of skew braces. At last, we state a characterisation of a recently defined class of skew braces~\cite{BNY22}, called $\gamma$-homomorphic skew braces, which bears a strong resemblance to a characterisation of bi-skew braces by Caranti. This resemblance is further emphasised when we discuss two slightly different constructions. One yields a new construction of $\gamma$-homomorphic skew braces and the other is a new way to obtain examples of bi-skew braces described by Childs.

Section~\ref{sec: classification} contains two classification results. We first prove an upper bound on the right nilpotency class of braces with multiplicative group isomorphic to $\Z^n$. In particular, we recover the known result that such a brace is trivial if $n=1$ and the new result that it is a bi-skew brace if $n=2$. Secondly, we use bi-skew braces to prove an open problem posed by L.~Vendramin concerning the classification of skew braces with a multiplicative group isomorphic to $\Z$.

Section~\ref{sec: bi and YBE} starts with a short summary concerning the connection between skew braces and set-theoretic solutions of the Yang--Baxter equation. We then give new results which show that the property of being a bi-skew brace can still be recognised when we look at the associated set-theoretic solution of the Yang--Baxter equation. On the other hand, we show with a counterexample that the associated solutions to both skew brace structures of a bi-skew brace can not be related in a direct way. We also study when the structure group of a solution is a bi-skew brace.

In section~\ref{sec: brace blocks} we investigate brace blocks. We start with a general characterisation of brace blocks on a given group. It is only when we add an extra condition that we obtain a more manageable characterisation from which we can construct new brace blocks. Nonetheless, we illustrate that this more restrictive characterisation still covers all known constructions of brace blocks in literature. This also allows to work with abelian groups, a case where most known constructions only yield trivial examples. We further give two new concrete constructions of brace blocks using rings and semidirect products.

\section{Preliminaries}\label{sec: prelim}

We begin with the definition of a skew left brace.
\begin{definition}
	A \emph{skew left brace} $A=(A,\cdot,\circ)$ is a set $A$ together with group structures $(A,\cdot)$ and $(A,\circ)$ such that for all $a,b,c\in A$,
	\begin{equation*}
	    a\circ (b\cdot c)=(a\circ b)\cdot a^{-1}\cdot (a\circ c).
	\end{equation*}
    Here $a^{-1}$ denotes the inverse of $a$ in $(A,\cdot)$. 
\end{definition}
\begin{remark}
    From this definition it is already clear that in this work sets might be endowed with more than one group structure. When there is possible confusion, we specify the group structure when well-known notations concerning groups are used. For example, the centre of $(A,\cdot)$ is denoted by $Z(A,\cdot)$, and the automorphism group of $(A,\circ)$ is denoted by $\Aut(A,\circ)$.
\end{remark}
In a natural way, also \emph{right} and \emph{two-sided skew braces} can be defined. However, for simplicity, we talk about \emph{skew braces} when actually meaning skew left braces. A skew brace $(A,\cdot,\circ)$ is called a \emph{brace} if $(A,\cdot)$ is an abelian group. Given a skew brace $(A,\cdot,\circ)$, we call $(A,\cdot)$ the \emph{additive group}; this somewhat ambiguous notation follows from the fact that for braces this group was originally denoted by $(A,+)$ and this notation is still common for skew braces. The group $(A,\circ)$ is called the \emph{multiplicative group}. By $\overline{a}$ we denote the inverse of an element $a\in A$ with respect to the multiplicative group. For $n\in \Z$ and $a\in A$,  we use $a^n$ for the $n$-th power of $a$ in $(A,\cdot)$ and $a^{\circ n}$ for the $n$-th power in $(A,\circ)$. It is easily proved that the neutral elements of the additive and multiplicative group coincide; this common neutral element is denoted by $1$. Given two skew braces $A$ and $B$, a map $f\colon A\to B$ is a \emph{homomorphism of skew braces} if both $f(a\cdot b)=f(a)\cdot f(b)$ and $f(a\circ b)=f(a)\circ f(b)$ hold for all $a,b\in A$. \emph{Isomorphisms} and \emph{automorphisms} are defined accordingly, and the group of automorphisms of a skew brace $(A,\cdot,\circ)$ is denoted by $\Aut(A,\cdot,\circ)$.

One way to construct skew braces is to start from any group $(G,\circ)$ and define $a\cdot b=a\circ b$ for all $a,b\in G$. This is called the \emph{trivial skew brace} on $(G,\circ)$, for which we use the notation $\Triv(G,\circ)$. If we instead define $a\cdot b=b\circ a$, then we obtain the \emph{almost trivial skew brace} on $(G,\circ)$, denoted by $\opTriv(G,\circ)$. When the group operation on $G$ is clear, also $\Triv(G)$ and $\opTriv(G)$ are used.

With each element $a$ of a skew brace $(A,\cdot,\circ)$ we associate the bijective map 
\begin{equation*}
    \gamma(a)\colon A\to A, \quad b\mapsto \lexp{b}{\gamma(a)}=a^{-1}\cdot (a\circ b),
\end{equation*}
which we write as a left exponent. Note that this map is also denoted by $\lambda_a$ in literature. 
By \cite[Proposition 1.9]{GV17}, this gives a well-defined group homomorphism $\gamma\colon (A,\circ)\to \Aut(A,\cdot)$.
Note that $(A,\circ)$ is fully determined by $(A,\cdot)$ and the function $\gamma$, as 
\begin{equation}\label{eq: circ wrt gamma}
    a\circ b=a\cdot \lexp{b}{\gamma(a)}.
\end{equation}
It follows that any group $(A,\cdot)$ and map $\gamma:A\to \Aut(A,\cdot)$ such that for all $a,b\in A$, 
\begin{equation}\label{eq: gamma function}
    \gamma(a\cdot \lexp{b}{\gamma(a)})=\gamma(a)\gamma(b),
\end{equation} 
determine a skew brace $(A,\cdot,\circ)$, where $(A,\circ)$ is given by \eqref{eq: circ wrt gamma}. Any map $\gamma$ satisfying \eqref{eq: gamma function} is called a \emph{gamma function} on $(A,\cdot)$.
\begin{definition}
Let $A$ be a skew brace.
\begin{enumerate}
    \item A \emph{left ideal} of $A$ is a subgroup $I$ of $(A,\cdot)$ such that $\lexp{I}{\gamma(a)}\subseteq I$ for all $a\in A$.
    \item An \emph{ideal} of $A$ is a left ideal $I$ which is moreover normal in $(A,\cdot)$ and $(A,\circ)$.
\end{enumerate}
\end{definition}
From \eqref{eq: circ wrt gamma} we immediately see that (left) ideals of a skew brace $A$ are also subgroups of $(A,\circ)$, so in particular they are subskew braces of $A$. Ideals are precisely the appropriate substructures needed to define the quotient of a skew brace by this substructure. Indeed, given an ideal $I$ of $A$, we have $a\cdot I=I\cdot a=a\circ I=I\circ a$ for all $a\in A$. We can thus construct the \emph{quotient skew brace} $A/I$ in the natural way. An example of an ideal is the \emph{socle} of a skew brace $A$, defined as $\Soc(A)=\ker (\gamma)\cap Z(A,\cdot)$. Another example is given by the kernel of any skew brace homomorphism. 

For $a,b\in A$, we define $a*b=\lexp{b}{\gamma(a)}\cdot b^{-1}=a^{-1}\cdot (a\circ b)\cdot b^{-1}$. One way to think about this new operation is as a commutator which measures how close $\lexp{b}{\gamma(a)}$ is to $b$. Equivalently, $a*b$ can be seen as a measure in difference between $a\cdot b$ and $a\circ b$. Indeed, it is clear that these coincide if and only if $a*b=1$. Yet another way to see this operation becomes apparent when we look at the correspondence between two-sided braces and Jacobson radical rings, where $*$ yields precisely the ring operation of the associated Jacobson radical ring \cite{Rum07a}. For subsets $X,Y\subseteq A$ we define $X*Y$ as the subgroup of $(A,\cdot)$ generated by 
\begin{equation*}
    \{x*y\mid x\in X \text{ and } y\in Y\}.
\end{equation*}
Taking $X=Y=A$ we obtain $A^2=A*A$, which is an ideal of $A$ with the property that $A/A^2$ is a trivial skew brace. 

We can associate with every skew brace $(A,\cdot, \circ)$ its \emph{opposite skew brace} $A_\op=(A,\cdot_\op,\circ)$ where $a\cdot_\op b=b\cdot a$. This construction was first considered by A.~Koch and P.~J.~Truman \cite{KT20a}. It is now clear where the notation of the almost trivial skew brace on a given group $(A,\circ)$ comes from, as $\opTriv(A,\circ)=\Triv(A,\circ)_\op$. In general, we use ``$\op$'' to denote that we consider a known construction in the opposite skew brace. For example, $\gamma_\op$, respectively $*_\op$, is the gamma function, respectively $*$-operation, associated with $A_\op$. Concretely, for all $a,b\in A$,
\begin{gather*}
    \lexp{b}{\gamma_\op(a)}=(a\circ b)\cdot a^{-1},\\
    a*_\op b=\lexp{b}{\gamma_\op(a)} \cdot_\op b^{-1}=b^{-1}\cdot (a\circ b)\cdot a^{-1}.
\end{gather*}
Note that $A=A_\op$ if and only if $A$ is a brace. It is still possible that $A\cong A_\op$ when $A\neq A_\op$, although this is generally not the case. The internal structures of $A$ and $A_\op$ are strongly related however, as the following result, whose proof is immediate, shows.

\begin{proposition}\label{prop: ideals of op}
    Let $A$ be a skew brace. Then the ideals of $A$ and $A_\op$ coincide.
\end{proposition}
As a concrete application of Proposition~\ref{prop: ideals of op} we observe that $A_\op^2=A*_\op A$ is an ideal of $A$.

\begin{definition}
	A skew brace $(A,\cdot,\circ)$ is a \emph{bi-skew brace} if also $(A,\circ,\cdot)$ is a skew brace.
\end{definition}
Given a bi-skew brace $A=(A,\cdot,\circ)$ we use the notation $A_{\leftrightarrow}=(A,\circ,\cdot)$. The gamma function $\gamma_\leftrightarrow$ associated with $A_{\leftrightarrow}$ is given by $\gamma_\leftrightarrow(a)=\gamma(a)^{-1}$ and its $*$-operation is denoted by $*_\leftrightarrow$. The following characterisation of bi-skew braces using the gamma function is a slight reformulation of~\cite[Theorem 3.1]{Car20}.
\begin{theorem}\label{theorem: characterizations bi caranti}
    Let $A$ be a skew brace. Then the following are equivalent:
    \begin{enumerate}
        \item $A$ is a bi-skew brace.
        \item $\gamma:(A,\cdot)\to \Aut(A,\cdot)$ is a group antihomomorphism.
        \item $A^2_\op$ is contained in $\ker (\gamma)$.
    \end{enumerate}
\end{theorem}
\begin{remark}\label{rem: extra condition bi}
    We remark that another equivalent condition to the ones in Theorem~\ref{theorem: characterizations bi caranti} is that $A^2_\op*A=\{1\}$. Indeed, this follows from the fact that $a*b=1$ if and only if $\lexp{b}{\gamma(a)}=b$. Also note that in particular $A^2_\op$ is a trivial skew brace in this case.
\end{remark}
We recall the following notion, introduced as $\lambda$-homomorphic skew braces by V.~G. Bardakov, M.~V.~Neshchadim, and M.~K.~Yadav~\cite{BNY22}.
\begin{definition}
	A skew brace $A$ is \emph{$\gamma$-homomorphic} if $\gamma\colon (A,\cdot)\to \Aut(A,\cdot)$ is a group homomorphism.
\end{definition}

From Theorem~\ref{theorem: characterizations bi caranti} we recover the following result; see~\cite[Lemma 3.7]{Car20}.
\begin{lemma}\label{lemma: two implies three}
	Let $A$ be a skew brace. Then any two of the following statements imply the third:
\begin{itemize}
    \item $A$ is $\gamma$-homomorphic.
    \item $A$ is a bi-skew brace. 
    \item $\gamma(A)$ is abelian. 
\end{itemize}
\end{lemma}

To conclude this section, we move our attention to solubility and nilpotency properties of skew braces. Define inductively  $A^{(1)}=A^1=A^{[1]}=A_1=A$, and
\begin{align*}
	A^{(n)}&=A^{(n-1)}\ast A,\\
	A^{n}&=A\ast A^{n-1},\\
	A^{[n]}&=\left\langle\bigcup_{i=1}^{n-1} A^{[i]}*A^{[n-i]}\right\rangle,\\
	A_n&=A_{n-1}\ast A_{n-1}
\end{align*} for all $n\ge 2$, where on the third line we mean the subgroup generated in $(A,\cdot)$. Here $A^{(n)}$ is an ideal of $A$, $A^{n}$ and  $A^{[n]}$ are left ideals of $A$, and $A_n$ is a subskew brace of $A$; see~\cite{CSV19,KSV21}.

\begin{definition}
	A skew brace $A$ is 
	\begin{itemize}
		\item \emph{right nilpotent} if there exists $n \ge 1$ such that $A^{(n)}=\{1\}$; if $n$ is minimal with this property, then we call $n-1$ the \emph{right nilpotency class} of $A$. 
		\item \emph{left nilpotent} if there exists $n \ge 1$ such that $A^{n}=\{1\}$; if $n$ is minimal with this property, then we call $n-1$ the \emph{left nilpotency class} of $A$. 
		\item \emph{strongly nilpotent} if there exists $n \ge 1$ such that $A^{[n]}=\{1\}$; if $n$ is minimal with this property, then we call $n-1$ the \emph{strong nilpotency class} of $A$. 
		\item \emph{soluble}  if there exists $n \ge 1$ such that $A_{n}=\{1\}$; if $n$ is minimal with this property, then we call $n-1$ the \emph{solubility class} of $A$. 
	\end{itemize}
\end{definition}
For two-sided braces, the above notions all coincide with nilpotency of the associated Jacobson radical ring. Note that when the skew brace is almost trivial, then $a*b$ is the commutator of $a^{-1}$ and $b$, so the above notions coincide with their group theoretical counterparts.
We have the following connection between left, right and strongly nilpotent skew braces, proved in~\cite[Theorem 12]{Smo18} for braces and generalised in ~\cite[Theorem 2.30]{CSV19} to skew braces.
\begin{theorem}\label{theorem: nilpotency}
    Let $A$ be a skew brace. Then the following are equivalent:
    \begin{enumerate}
        \item $A$ is left and right nilpotent.
        \item $A$ is strongly nilpotent.
    \end{enumerate}
\end{theorem}
We conclude with the following result, which was stated in~\cite[Proposition 2.4]{BCJO19} for braces but can easily be generalised to skew braces.
\begin{lemma}\label{lemma: soluble quotients and ideals}
	Let $A$ be a skew brace and let $I$ be an ideal. Then $A$ is a soluble skew brace if and only if $A/I$ and $I$ are soluble skew braces. 
\end{lemma}

\section{Structural results for bi-skew braces}\label{sec: structural result}

In this section, we deal with results on nilpotency and solubility of bi-skew braces. We start by relating the structure of a bi-skew brace $A$ with that of $A_\leftrightarrow$.

\begin{lemma}\label{lemma: ideals A and A leftrightarrow coincide}
Let $A$ be a bi-skew brace. Then the (left) ideals of $A$ and $A_{\leftrightarrow}$ coincide.
\end{lemma}
\begin{proof}
Let $I$ be a subskew brace of $A$. As $\gamma_{\leftrightarrow}(a)=\gamma(a)^{-1}=\gamma(\overline{a})$, we have that $I$ is mapped to itself by $\gamma(a)$ for all $a\in A$ if and only if $I$ is mapped to itself by $\gamma_{\leftrightarrow}(a)$ for all $a\in A$. 
\end{proof}

\begin{lemma}
Let $A$ be a bi-skew brace, and let $I$ be a left ideal of $A$. Then $A*I=A*_{\leftrightarrow}I$. If furthermore $I$ is an ideal, then $I*A=I*_{\leftrightarrow} A$.
\end{lemma}
\begin{proof}
Suppose that $I$ is a left ideal of $A$. Take $a\in A$ and $b\in I$. We have 
\begin{align}
    a*b&=\lexp{b}{\gamma(a)}\cdot b^{-1}=(\lexp{b}{\gamma(a)}\circ \overline{b}\circ b)\cdot b^{-1}\nonumber\\
    &=(\lexp{b}{\gamma(a)}\circ \overline{b})\cdot \lexp{b}{\gamma(\lexp{b}{\gamma(a)}\circ\overline{b})}\cdot b^{-1}\nonumber\\
    &=(\overline{a}*_\leftrightarrow b)\cdot ((\overline{a} *_\leftrightarrow b)*b).\label{eq: corrected lemma}
\end{align}
Hence $\overline{a}*_\leftrightarrow b=(a*b)\cdot ((\overline{a} *_\leftrightarrow b)*b)^{-1}\in A*I$, and thus $A*_\leftrightarrow I\subseteq A*I$. By a symmetric argument and Lemma \ref{lemma: ideals A and A leftrightarrow coincide}, we also obtain $A*I\subseteq A*_\leftrightarrow I$.

Suppose now that $I$ is an ideal. By \eqref{eq: corrected lemma} with $a\in I$ and $b\in A$ and Lemma~\ref{lemma: ideals A and A leftrightarrow coincide}, we get that $I*_\leftrightarrow A\subseteq  I*A$. Therefore the result follows by a symmetric argument.
\end{proof}

As a consequence, we derive the following propositions. 
\begin{proposition}
    Let $A$ be a bi-skew brace. Then $A$ is soluble of class $n$ if and only if $A_{\leftrightarrow}$ is soluble of class $n$.
\end{proposition}
\begin{proposition}
    Let $A$ be a bi-skew brace. Then $A$ is left, respectively right, nilpotent of class $n$ if and only if $A_{\leftrightarrow}$ is left, respectively right, nilpotent of class $n$.
\end{proposition}

We show now that we can check whether a bi-skew brace is right nilpotent or soluble just looking at a suitable group. A prominent role is played by Theorem~\ref{theorem: characterizations bi caranti} and Remark~\ref{rem: extra condition bi}. Also the following lemma is used throughout the rest of this section.
\begin{lemma}\label{lem: ker ideal if bi}
Let $A$ be a bi-skew brace. Then $\ker (\gamma)$ is an ideal of $A$.
\end{lemma}
\begin{proof}
It suffices to note that $\gamma\colon A\to \opTriv(\Aut(A,\cdot))$ is well-defined skew brace homomorphism. 
\end{proof}
\begin{theorem}\label{theorem: A nilpotent if im gamma is nilpotent} 
Let $A\neq \{1\}$ be a bi-skew brace. Then $A$ is right nilpotent of class $n+1$ if and only if $\gamma(A)$ is a nilpotent group of class $n$.
\end{theorem}
\begin{proof}
As $A$ is a bi-skew brace, $\ker(\gamma)$ is an ideal, trivial as a skew brace. Because $A\neq \{1\}$, it is clear that $A$ is right nilpotent of class $n+1$ if and only if $A/\ker(\gamma)$ is right nilpotent of class $n$. 
    
As $A_\op^2\subseteq \ker(\gamma)$, we know that $A/\ker(\gamma)$ is an almost trivial skew brace. In particular, $A/\ker(\gamma)$ is right nilpotent of class $n$ if and only if the group $(A/\ker(\gamma),\circ)$ is nilpotent of class $n$, and $(A/\ker(\gamma),\circ)$ is clearly isomorphic to $\gamma(A)$. 
\end{proof}
\begin{corollary}
    Let $A$ be a bi-skew brace such that $(A,\cdot)$ or $(A,\circ)$ are nilpotent. Then $A$ is right nilpotent.
\end{corollary}
\begin{proof}
    It suffices to note that $\gamma(A)$ is a quotient of both $(A,\cdot)$ and $(A,\circ)$, and then to apply Theorem~\ref{theorem: A nilpotent if im gamma is nilpotent}. 
\end{proof}
\begin{corollary}
Let $A$ be a bi-skew brace. Then $A$ is left nilpotent if and only if $A$ is strongly nilpotent.
\end{corollary}
\begin{proof}
By Theorem~\ref{theorem: nilpotency}, it suffices to show that in this case left nilpotency implies right nilpotency. If $A$ is left nilpotent, then so is the skew brace $A/\ker (\gamma)$. But as $A/\ker (\gamma)$ is almost trivial, we find that this is equivalent to the group $(A/\ker(\gamma),\circ)\cong \gamma(A)$ being nilpotent. The result then follows from Theorem~\ref{theorem: A nilpotent if im gamma is nilpotent}.
\end{proof}
\begin{proposition}\label{prop: A soluble if im gamma is soluble}
Let $A$ be a bi-skew brace. Then $A$ is a soluble skew brace if and only if $\gamma(A)$ is a soluble group.
\end{proposition}
\begin{proof}
    As $A$ is a bi-skew brace, $\ker(\gamma)$ is an ideal, trivial as a skew brace. Therefore $A$ is soluble if and only if $A/\ker(\gamma)$ is soluble, by Lemma~\ref{lemma: soluble quotients and ideals}. 
    
    Now as $A_\op^2\subseteq \ker(\gamma)$, we know that $A/\ker(\gamma)$ is an almost trivial skew brace. In particular, $A/\ker(\gamma)$ is soluble if and only if the group $(A/\ker(\gamma),\circ)$ is soluble, and $(A/\ker(\gamma),\circ)$ is clearly isomorphic to $\gamma(A)$.
\end{proof}

We conclude this section by proving Byott's conjecture in the case of bi-skew braces. 
\begin{conjecture}[Byott's Conjecture]
	Let $A$ be a finite skew brace. If $(A,\cdot)$ is soluble, then $(A,\circ)$ is soluble. 
\end{conjecture}

\begin{theorem}
    Let $A$ be a bi-skew brace. Then $(A,\cdot)$ is soluble if and only if $(A,\circ)$ is soluble. Moreover, in this case also $A$ is soluble as a skew brace.
\end{theorem}
\begin{proof}
    Assume that $(A,\cdot)$ is soluble. As $A/A_\op^2$ is an almost trivial skew brace, $(A/A_\op^2,\cdot)\cong (A/A_\op^2,\circ)$, and in particular it follows from the assumption that both are soluble. Since $A^2_\op$ is a trivial skew brace, clearly $(A^2_\op,\cdot)\cong (A^2_\op,\circ)$. In particular, it once again follows from the assumption that both groups are soluble. We conclude that $(A/A_\op^2,\circ)$ and $(A_\op^2,\circ)$ are soluble, so $(A,\circ)$ is also soluble.
    
    The exact same argument proves the other implication. 
    
    To prove that in this case $A$ is also soluble as a skew brace, it suffices to note that $\gamma(A)$ is a soluble group as it is a quotient of $(A,\circ)$. The solubility of $A$ then follows by Proposition~\ref{prop: A soluble if im gamma is soluble}. 
\end{proof}
\begin{remark}
    The same idea can be used to prove Byott's conjecture for skew braces with a composition series where the factors are trivial or almost trivial skew braces, so in particular for soluble skew braces. This means that Byott's conjecture holds for several classes of skew braces studied in recent years, but as there exist finite simple skew braces which are neither trivial nor almost trivial~\cite{BCJO19}, this does not provide a general way to prove the conjecture.
\end{remark}

\subsection{Bi-skew braces and $\gamma$-homomorphic skew braces}
Recall from Theorem~\ref{theorem: characterizations bi caranti} that the gamma function of a bi-skew brace is an antihomomorphism with respect to the additive group. It is therefore to be expected that there is a connection with $\gamma$-homomorphic skew braces. Indeed, as we noted before these two notions coincide when $\gamma(A)$ is abelian. In the remainder of this section we further discuss some similarities.

We start with a theorem for $\gamma$-homomorphic skew braces similar to Theorem~\ref{theorem: characterizations bi caranti} and Remark~\ref{rem: extra condition bi}.

\begin{theorem}\label{thm:equivalentgamma}
Let $A$ be a skew brace. Then the following are equivalent:
\begin{enumerate}
    \item $A$ is $\gamma$-homomorphic.
    \item $A^2$ is contained in $\ker(\gamma)$.
    \item $A$ is right nilpotent of class at most $2$.
\end{enumerate}
\end{theorem}
\begin{proof}
Assume that $A$ is $\gamma$-homomorphic. Then for all $a,b\in A$,
\begin{equation*}
    \gamma(a*b)=\gamma(a)^{-1}\gamma(a)\gamma(b)\gamma(b)^{-1}=1.
\end{equation*}
It follows that $A^2\subseteq \ker(\gamma)$.
Conversely, assume that $\ker(\gamma)$ contains $A^2$. As $A/A^2$ is a trivial skew brace it follows for all $a,b\in A$ that $(a\cdot b)\circ A^2=(a\circ b)\circ A^2$. The assumption then implies \[\gamma(a\cdot b)=\gamma(a\circ b)=\gamma(a)\gamma(b).\]

The equivalence of the second and third condition is clear.
\end{proof}

Recall that a skew brace $A$ is \emph{metatrivial} if it is soluble of class at most 2, which by definition means that $A^2$ is a trivial skew brace. As a consequence of Theorem~\ref{thm:equivalentgamma}, we find a short proof of~\cite[Theorem 2.12]{BNY22}, as follows.
\begin{corollary}
Every $\gamma$-homomorphic skew brace is metatrivial.
\end{corollary}
\begin{proof}
Let $A$ be a $\gamma$-homomorphic skew brace. Then by Theorem~\ref{thm:equivalentgamma}, $A^{(3)}=\{1\}$, and this clearly implies that $A^2*A^2\subseteq A^2*A=\{1\}$.
\end{proof}
Note that the converse does not hold. 
\begin{example}
Let $A=\opTriv(\Sym_3)$, with $\Sym_3$ the symmetric group on 3 elements. Then $A$ is metatrivial because $\Sym_3$ is metabelian. As $\Sym_3$ is not nilpotent, it follows that $A$ is not right nilpotent and in particular not $\gamma$-homomorphic.
\end{example}

\begin{example}
	Recall that if $A$ is a Jacobson radical ring, then $A$ yields a two-sided brace~\cite{Rum07a}.
	By Theorem~\ref{theorem: characterizations bi caranti} (or equivalently, Theorem~\ref{thm:equivalentgamma}), we find that the corresponding brace is a bi-skew brace (or equivalently a $\gamma$-homomorphic skew brace) if and only if $A^{(3)}=A^3=\{1\}$, as already shown in~\cite[Proposition 4.1]{Chi19} when $A$ is finite or nilpotent.
\end{example}

We conclude the section by showing that one can use the semidirect product of skew braces to obtain two slightly different constructions, one yielding $\gamma$-homomorphic skew braces and one yielding bi-skew braces. For bi-skew braces it turns out that we find a different construction for examples which were already described by Childs.

We recall first the semidirect product of skew braces, as developed in~\cite[Corollary 2.36]{SV18}. If $A$ and $B$ are skew braces, an \emph{action} of $A$ on $B$ is a group homomorphism $\alpha\colon (A,\circ)\to \Aut(B,\cdot,\circ)$. If we have such an $\alpha$, whose action is written as a left exponent, then the set $A\times B$, with operations
\begin{align*}
	(a,b)\cdot(a',b')&=(a\cdot a',b\cdot b'),\\
	(a,b)\circ(a',b')&=(a\circ a',b\circ\lexpp{b'}{\alpha(a)}),
\end{align*}
is a skew brace, which we denote by $A\ltimes B$. 

\begin{example}\label{example: semidirect gamma}
	Let $G$ and $H$ be groups, and let $G$ act by automorphisms on $H$, with the action denoted by $\alpha$. This induces an action of the trivial skew brace $A=\Triv(G)$ on the trivial skew brace $B=\Triv(H)$. By the semidirect product construction, we find a skew brace $A\ltimes B$. Explicitly, 
	\begin{align*}
	(g,h)\cdot (g',h')&=(gg',hh'),\\
	(g,h)\circ(g',h')&=(gg',h\lexpp{h'}{\alpha(g)}).
    \end{align*}
We have recovered in this way~\cite[Example 1.4]{GV17}. Note that here the gamma function is $\gamma(g,h)=(\id,\alpha(g))$. In particular, $A\ltimes B$ is a $\gamma$-homomorphic skew brace, and it is a bi-skew brace if and only if $[G,G]\subseteq \ker(\alpha)$, as an immediate computation shows. 
\end{example}

\begin{example}\label{example: semidirect bi}
	Let $G$ and $H$ be groups, and let $G$ act by automorphisms on $H$, with the action denoted by $\alpha$. This induces an action of the almost trivial skew brace $A=\opTriv(G)$ on the trivial skew brace $B=\Triv(H)$. By the semidirect product construction, we find a skew brace $A\ltimes B$. Explicitly, 
	\begin{align*}
	(g,h)\cdot (g',h')&=(g'g,hh'),\\
	(g,h)\circ(g',h')&=(gg',h\lexpp{h'}{\alpha(g)}).
\end{align*}
The skew brace $A\ltimes B$ has gamma function $\gamma(g,h)=(\phi(g),\alpha(g))$, where $\phi(g)$ denotes conjugation by $g$. This immediately implies that $A\ltimes B$ is a bi-skew brace, already obtained in~\cite[Proposition 7.1]{Chi19} from the semidirect product of the groups $G$ and $H$. Moreover it is $\gamma$-homomorphic if and only if $[G,G]\subseteq \ker(\alpha)\cap Z(G)$. Note that when $G$ is abelian, this construction coincides with the one in Example~\ref{example: semidirect gamma}.
\end{example}

\section{Some classification results}\label{sec: classification}

We begin this section by showing that all the braces with multiplicative group isomorphic to $\Z^2$ are in fact bi-skew braces. We need the following result.
\begin{theorem}\label{theorem: Amberg}
Let $(A,\cdot,\circ)$ be a two-sided brace such that $(A,\circ)$ is finitely generated abelian of rank $n$. Then $(A,\cdot)$ is finitely generated abelian of rank $n$.
\end{theorem}
\begin{proof}
By~\cite[Theorem 3]{Wat68}, if $(A,\circ)$ is finitely generated abelian, then $(A,\cdot)$ is finitely generated. By~\cite[Theorem B]{AD95}, the ranks of $(A,\circ)$ and $(A,\cdot)$ coincide.
\end{proof}

\begin{proposition}\label{prop: mult group Zn nilpotent}
Let $A$ be a brace with multiplicative group isomorphic to $\Z^n$. Then $A$ is right nilpotent of class at most $n$.
\end{proposition}
\begin{proof}
By Theorem~\ref{theorem: Amberg}, $(A,\cdot)$ is finitely generated of rank $n$. Let $T$ be the (necessarily finite) torsion subgroup of $(A,\cdot)$. As $T$ is a characteristic subgroup of $(A,\cdot)$, it is a left ideal of $A$, so also a finite subgroup of $(A,\circ)$, and thus $T$ is trivial. It follows that $(A,\cdot)\cong \Z^n$. 

Now for a prime $p$, let $I_p\cong (p\Z)^n$ be the characteristic subgroup of $(A,\cdot)$ generated by all $p$-powers of elements. Then $A/I_p$ has order $p^n$, and therefore it is left nilpotent of class at most $n$ by~\cite[Corollary of Proposition 8]{Rum07a}. As $A/I_p$ is a two-sided brace it is also right nilpotent of class at most $n$, or equivalently, $A^{(n+1)}\subseteq I_p$. We conclude that
\begin{equation*}
	A^{(n+1)}\subseteq \bigcap_{\text{$p$ prime}}I_p=\{1\}. \qedhere
\end{equation*}
\end{proof}

We immediately recover~\cite[Theorem 5.5]{CSV19}, and we find the result we have claimed.

\begin{corollary}\label{cor: Triv(Z) is only abelian brace with multiplicative group Z}
Let $A$ be a brace with multiplicative group isomorphic to $\Z$. Then $A$ is a trivial skew brace.
\end{corollary}
\begin{corollary}
Let $A$ be a brace with multiplicative group isomorphic to $\Z^2$. Then $A$ is a bi-skew brace.
\end{corollary}
\begin{proof}
It follows from Proposition~\ref{prop: mult group Zn nilpotent} that $A$ is right nilpotent of degree at most $2$. The statement then follows from Theorem~\ref{theorem: characterizations bi caranti} and the fact that $A_\op=A$.
\end{proof}

Motivated by Corollary~\ref{cor: Triv(Z) is only abelian brace with multiplicative group Z}, we now want to classify all the skew braces with multiplicative group isomorphic to $\Z$, as asked in~\cite[Problem 2.27]{Ven19}. Let $(A,\cdot)$ be an infinite cyclic group, with generator $x$. We can define the following operation on $A$:
\begin{equation*}
	x^i\circ x^j=x^{i+(-1)^ij}.
\end{equation*}
Then, as shown in the proof of~\cite[Proposition 6]{Rum07b}, the operation $\circ$ is the unique one such that $(A,\cdot,\circ)$ is a nontrivial brace, and $(A,\circ)$ is isomorphic to the infinite dihedral group
\begin{equation*}
	\langle x,y\mid y^2=1,yxy=x^{-1}\rangle\cong C_2\ltimes \Z. 
\end{equation*}

We can easily see that $(A,\cdot,\circ)$ is a bi-skew brace. Moreover, there are just two group automorphisms of $(A,\circ)$, namely the identity and the inversion, and it is easily verified that both are also automorphisms of the skew brace $(A,\cdot, \circ)$ and therefore also of $(A,\circ,\cdot)$.

We claim that $(A,\circ,\cdot)$ is not isomorphic to its opposite skew brace. They are clearly not equal, as $(A,\circ)$ is not abelian. Therefore, the only candidate for an isomorphism is given by the inversion automorphism of $(A,\cdot)$ which also induces an automorphism of $(A,\circ)$. If this yields an isomorphism of skew braces, then it would be both an automorphism and antiautomorphism of $(A,\circ)$ which implies that $(A,\circ)$ is abelian. As $(A,\circ)$ is isomorphic to the infinite dihedral group, this is a contradiction.

In the remainder of this section we prove that the two skew braces with infinite cyclic multiplicative group above are in fact the only nontrivial ones.

\begin{lemma}\label{lem: abelian mult group X*Y=Y*opX}
Let $B$ be a skew brace with an abelian multiplicative group. Then for all $X,Y\sub B$, the equality $X*Y=Y*_\op X$ holds.
\end{lemma}
\begin{proof}
It suffices to note that for all $a,b\in B$,
\begin{equation*}
	a* b=a^{-1}\cdot (a\circ b)\cdot b^{-1}=a^{-1}\cdot (b\circ a)\cdot b^{-1}=b*_\op a.\qedhere
\end{equation*}
\end{proof}
\begin{theorem}\label{theorem: classification skew brace Z mult group}
Let $(A,\circ)=\{x^{\circ i}\mid i\in \Z\}$ be an infinite cyclic group. If $A=(A,\cdot,\circ)$ is a skew brace, then the additive operation is given by one of the following equalities:
\begin{align}
    x^{\circ i}\cdot x^{\circ j}&=x^{\circ (i+ j)},\label{eq: mult1}\\
    x^{\circ i}\cdot x^{\circ j}&=	x^{\circ (i+(-1)^ij)}\label{eq: mult2},\\
    x^{\circ i}\cdot x^{\circ j}&=	x^{\circ(j+(-1)^ji)}.\label{eq: mult3}
\end{align}
\end{theorem}
\begin{proof}
If $(A,\cdot)$ is abelian, then $\cdot$ is given by \eqref{eq: mult1} by Corollary~\ref{cor: Triv(Z) is only abelian brace with multiplicative group Z}. 

From now on, we assume that $(A,\cdot)$ is not abelian. As $A$ is two-sided, it follows from~\cite[Lemma 4.5]{Nas19} that $(A^2,\cdot)$ is abelian. In particular, $A^2\ne A$. Moreover, note that $A^2\ne 1$, otherwise $A$ would be trivial, so $(A,\cdot)$ would be abelian. We deduce that there exists $n\ge 2$ such that $A^2=\{x^{\circ nk}\mid k\in \Z\}$.

 As $A^2$ is a brace with multiplicative group isomorphic to $\Z$, it follows from Corollary~\ref{cor: Triv(Z) is only abelian brace with multiplicative group Z} that $A^2$ is a trivial skew brace. Because $(A/A^2,\circ)\cong C_n$, we find that $A/A^2\cong \Triv(C_n)$. 
 
 As $x$ generates $(A,\circ)$, its equivalence class in $A/A^2$ generates $(A/A^2,\circ)$, and therefore also $(A/A^2,\cdot)$. If we take $a\in A^2$ to be a generator of $(A^2,\cdot)$, then $(A,\cdot)$ is generated by $a$ and $x$. Denote by $\psi$ the inner automorphism on $(A^2,\cdot)$ induced by $x$ in $(A,\cdot)$. As $(A,\cdot)$ is not abelian, $\psi$ is not trivial, so necessarily $\psi$ is the inversion automorphism. 
 
 Likewise, $\gamma(x)$ restricts to an automorphism of $(A^2,\cdot)$, which is either the identity or equals $\psi$.
 \begin{enumerate}
  	\item If $\gamma(x)$ is the identity on $A^2$, then using Lemma~\ref{lem: abelian mult group X*Y=Y*opX} we find $A^2*_\op A=A*A^2=\{1\}$, so $A_\op$ is a bi-skew brace. This means that $(A,\circ,\cdot_\op)$ is a nontrivial skew brace on $(A,\circ)$, so necessarily $\cdot$ is given by \eqref{eq: mult3}. 
    \item If $\gamma(x)$ restricts to the inversion automorphism on $(A^2,\cdot)$, and therefore is equal to $\psi$ on $A^2$, then $\gamma_{\op}(x)$ equals $\psi^2=\id$ on $A^2$. Using Lemma~\ref{lem: abelian mult group X*Y=Y*opX} we find $A^2_\op*A=A*_\op A^2=\{1\}$. So $(A,\circ,\cdot)$ is a non-trivial skew brace on $(A,\circ)$, which implies that $\cdot$ is given by \eqref{eq: mult2}.
 	 \qedhere
 \end{enumerate} 
\end{proof}

\section{Bi-skew braces and solutions of the Yang--Baxter equation}\label{sec: bi and YBE}

The interest in set-theoretic solutions of the Yang--Baxter equation, as a simplification of its linear solutions, goes back to V.~G.~Drinfel'd \cite{Dri92}. Compared to the linear version, set-theoretic solutions are easier to study and classify. Nonetheless, set-theoretic solutions can be linearised and are also omnipresent in the study of link and knot invariants~\cite{NV06}. In this section we investigate connections between bi-skew braces and associated set-theoretic solutions of the Yang--Baxter equation. Before doing so, we give a short summary of all the necessary notions involved.

A \emph{set-theoretic solution of the Yang--Baxter equation} is a pair $(X,r)$ with $X$ a nonempty set and $r\colon X^2\to X^2$ a bijective map satisfying the braid equation
\begin{equation*}
    r_1r_2r_1=r_2r_1r_2,
\end{equation*}
where $r_1=r\times \id_X$ and $r_2=\id_X\times r$. On every nonempty set $X$, the map $(x,y)\mapsto (y,x)$ satisfies the braid equation; this solution is said to be \emph{trivial}. We say that a solution is \emph{nondegenerate} if the maps $\sigma_x,\tau_x\colon X\to X$ defined by $r(x,y)=(\sigma_x(y),\tau_y(x))$ are bijective. In the rest of this section, we will shortly refer to a nondegenerate set-theoretic solution of the Yang--Baxter equation as a \emph{solution}. Furthermore, we say that a solution is \emph{involutive} if $r^2=\id$. Given a solution $(X,r)$, it is easily checked that also $(X,r^{-1})$ is a solution. We define $\hat{\sigma},\hat{\tau}\colon X\to X$ by $r^{-1}(x,y)=(\hat{\sigma}_x(y),\hat{\tau}_y(x))$. 

Given two solutions $(X,r)$ and $(Y,s)$, we say that a map $f\colon X\to Y$ is a \emph{homomorphism of solutions} if $(f\times f)r=s(f\times f)$. One can prove that in this case the image of $f$ is a subsolution of $(Y,s)$, meaning that $s$ restricts to $f(X)\times f(X)$. A bijective homomorphism of solutions is called an \emph{isomorphism of solutions}.

Given a skew brace $A$, define \begin{equation*}
    r_A\colon A^2\to A^2, \quad (a,b)\mapsto (\lexp{b}{\gamma(a)},\overline{\lexp{b}{\gamma(a)}}\circ a\circ b).
\end{equation*}
Then the pair $(A,r_A)$ is a solution~\cite{LYZ00,GV17} and $(A_\op, r_{A_\op})$ is its inverse solution ~\cite{KT20a}. This construction is functorial: a skew brace homomorphism $f\colon A\to B$ induces a homomorphism of solutions $f\colon (A,r_A)\to (B,r_B)$.

Following~\cite{ESS99} we define the \emph{structure group} of a solution $(X,r)$ as 
\begin{equation*}
    G(X,r)=\langle X\mid x\circ y=\sigma_x(y)\circ \tau_y(x) \text{ for all }x,y\in X\rangle.
\end{equation*}
For each solution $(X,r)$, one can also define its \emph{derived structure group} as
\begin{equation*}
    A(X,r)=\langle X\mid x\cdot y=y\cdot \sigma_y\hat{\sigma}_y(x)\rangle.
\end{equation*}
It is possible to construct a bijection between $A(X,r)$ and $G(X,r)$ such that, transferring the group structure of $A(X,r)$ to $G(X,r)$, one obtains a skew brace $(G(X,r),\cdot,\circ)$. This is a brace if and only if $(X,r)$ is involutive; see~\cite{Sol00,LV19}. Under this bijection, the generator of $A(X,r)$ corresponding to some element $x\in X$ is mapped to the generator of $G(X,r)$ corresponding to $x$. We define $\iota\colon X\to G(X,r)$ as the canonical map sending each element to its corresponding generator. We therefore find that $\iota(X)\subseteq (G(X,r),\cdot,\circ)$ generates both the additive and multiplicative group. Moreover, the skew brace $(G(X,r),\cdot,\circ)$ satisfies the property that $\iota$ is a homomorphism of solutions $\iota\colon (X,r)\to (G(X,r),r_{G(X,r)})$. In particular, this means that for all $x,y\in X$ the equality $\iota(\sigma_x(y))=\lexpp{\iota(y)}{\gamma(\iota(x))}$ holds in $G(X,r)$. Similarly, it follows that $\iota(\hat{\sigma}_x(y))=\lexpp{\iota(y)}{\gamma_\op(\iota(x))}$ holds in $G(X,r)$. A solution is \emph{injective} if $\iota$ is an injective map. An involutive solution is always injective. To any solution $(X,r)$ one can associate its \emph{injectivization} $\Inj(X,r)$, which is the image of the homomorphism $\iota\colon(X,r)\to (G(X,r),r_{G(X,r)})$. It is clear that $G(\Inj(X,r))$ and $G(X,r)$ are isomorphic as skew braces and therefore $\Inj(X,r)$ is indeed an injective solution.

Following~\cite{Bac18,CJKVAV22}, we can associate to any solution $(X,r)$ another group called the \emph{permutation group}, defined as 
\begin{equation*}
    \mathcal{G}(X,r)=\langle (\sigma_x,\tau_x^{-1})\mid x\in X\rangle \subseteq \Perm(X)\times \Perm(X).
\end{equation*}
There exists a unique surjective group homomorphism $\pi\colon G(X,r)\to \mathcal{G}(X,r)$ satisfying $\pi(x)=(\sigma_x,\tau_x^{-1})$. It can be shown that the kernel of this map $\pi$ is an ideal of $(G(X,r),\cdot,\circ)$, hence there is a natural skew brace structure induced on $\mathcal{G}(X,r)$ such that $\pi$ is a skew brace homomorphism. If $(X,r)$ is injective, then $\ker (\pi)=\Soc(G(X,r))$. In particular, note that for any involutive (hence injective) solution $(X,r)$, we find that $\ker (\pi)=\ker (\gamma)$. As the construction of the solution associated to a skew brace is functorial, we find a homomorphism of solutions $\pi\iota\colon (X,r)\to (\G(X,r),r_{\G(X,r)})$. The image of this homomorphism is called the \emph{retract} of $(X,r)$, denoted by $\Ret(X,r)$. It can be verified that $\Ret(X,r)$ can also be obtained as the induced solution on the equivalence classes of the equivalence relation given by
\begin{equation*}
    x\sim y\iff \sigma_x=\sigma_y\text{ and }\tau_x=\tau_y,
\end{equation*}
which is its original definition in literature.

We now move on to our first result.

\begin{proposition}\label{proposition: new condition bi}
	Let $A$ be a skew brace. Then $A$ is a bi-skew brace if and only if for all $a,b\in A$,
	\begin{equation}\label{eq: 6.1}
		\gamma(\lexp{b}{\gamma_{\op}(a)})=\gamma(b).
	\end{equation}
\end{proposition}
\begin{proof}
	If $A$ is a bi-skew brace, then the assertion follows from Theorem~\ref{theorem: characterizations bi caranti}.

	For the other implication, suppose that \eqref{eq: 6.1} holds. For all $a,b\in A$, we have
 \begin{align*}
     \gamma(a\cdot b)&=\gamma(\lexp{a}{\gamma_\op(b)\gamma_\op(\overline{b})}\cdot b)\\
     &=\gamma(b\circ \lexp{a}{\gamma_\op(\overline{b})})\\
     &=\gamma(b)\gamma(\lexp{a}{\gamma_\op(\overline{b})})\\
     &=\gamma(b)\gamma(a).
 \end{align*}
 Hence, again, by Theorem~\ref{theorem: characterizations bi caranti}, $A$ is a bi-skew brace.
\end{proof}

The following is a straightforward corollary.
\begin{proposition}\label{prop: condition solution bi}
    Let $A$ be a skew brace. Then $A$ is a bi-skew brace if and only if its associated solution $(A,r_A)$ satisfies, for all $x,y\in A$,
    \begin{equation*} \sigma_{\hsigma_x(y)}=\sigma_y.
    \end{equation*}
\end{proposition}

As a result, the information whether $A$ is a bi-skew brace is not lost when one only considers its associated solution. Next, it is natural to ask whether, if we know that $A$ is a bi-skew brace, it is possible to recover the associated solution of $A_{\leftrightarrow}$ from the associated solution of $A$. The following example shows that this is in general not possible, as we construct nonisomorphic bi-skew braces $A$ and $B$ such that the associated solutions are isomorphic, but the solutions associated to $A_{\leftrightarrow}$ and $B_{\leftrightarrow}$ are not isomorphic.
\begin{example}
    Let $(G,\cdot)=C_2\times C_8$, with $C_2=\langle x\rangle$, and let $\psi_1\colon C_2\to \Aut(G)$ be the group homomorphism mapping $x$ to the inversion automorphism of $G$. By the semidirect product construction (see section~\ref{sec: structural result}), we find a bi-skew brace $A$ on the set $C_2\times G$, whose associated solutions can easily be computed as follows:
    \begin{align*}
    r_{A}((x^i,g),(x^j,h))&=((x^j,\lexp{h}{\psi_1(x)^i}),(x^i,\lexp{g}{\psi_1(x)^{j}})),\\
        r_{A_{\leftrightarrow}}((x^i,g),(x^j,h))
        &=((x^j,\lexp{h}{\psi_1(x)^i}),(x^i,g\cdot h\cdot\lexpp{h^{-1}}{\psi_1(x)^{i}})).
    \end{align*}
    In particular, ${\tau_{\leftrightarrow,(x^j,h)}}(1,g)=(1,g)$ and ${\tau_{\leftrightarrow,(x^j,h)}}(x,g)=(x,g\cdot h^2)$. Here $\tau_\leftrightarrow$ is the usual map associated with the solution $(A,r_{A_{\leftrightarrow}})$. Note that if $h\in G$ is an element of order 8, then $\tau_{\leftrightarrow,(x^j,h)}$ has order 4.
    
    Now take $(H,\cdot)=C_2^4$, and let $\psi_2\colon C_2\to \Aut(H)$, where still $C_2=\langle x\rangle$, be the map which sends $x$ to the automorphism interchanging the first two and the last two coordinates of $H$. In the same way as before we can then obtain a bi-skew brace $B$, and two associated solutions:
    \begin{align*}
        r_{B}((x^i,g),(x^j,h))&=((x^j,\lexp{h}{\psi_2(x)^i}),(x^i,\lexp{g}{\psi_2(x)^{j}})),\\
        r_{B_{\leftrightarrow}}((x^i,g),(x^j,h))&=((x^j,\lexp{h}{\psi_2(x)^{i}}),(x^i,g\cdot h\cdot\lexp{h}{\psi_2(x)^{i}})).
    \end{align*}
    In this case once again, we find that ${\tau_{\leftrightarrow,(x^j,h)}}(1,g)=(1,g)$ and $\tau_{\leftrightarrow,(x^j,h)}(x,g)=(x,g\cdot h\cdot\lexp{h}{\psi_2(x)^{i}})$. One easily checks that all $\tau$-maps associated to $ r_{B_{\leftrightarrow}}$ have either order one or two. Therefore, $ r_{A_{\leftrightarrow}}$ can not be isomorphic to $ r_{B_{\leftrightarrow}}$.
    
    The cycle structures of $\psi_1(x)$ and $\psi_2(x)$ are the same; they are both of order two and fix four points. Therefore, there exists a bijection $\theta\colon G\to H$ such that $\theta \psi_1(x)=\psi_2(x)\theta$, and in particular the bijection 
    \begin{align*}
    	C_2\times G\to C_2\times H,\quad  
    	(x^i,g)\mapsto (x^i,\theta(g))
    \end{align*} gives an isomorphism between the solutions $r_{A}$ and $r_{B}$. 
\end{example}
We now deal with the opposite situation, where we start with a given solution and ask whether the skew brace on the structure group is a bi-skew brace.
\begin{theorem}\label{theorem: G(X r) is bi-skew}
    Let $(X,r)$ be an injective solution. Then $G(X,r)$ is a bi-skew brace if and only if $\sigma_{\hsigma_x(y)}=\sigma_y$ for all $x,y\in X$.
\end{theorem}
\begin{proof}
    The implication from left to right is a consequence of Proposition~\ref{prop: condition solution bi} and the fact that $(X,r)$ is injective.
    
    Now assume that for all $x,y\in X$, we have that $\sigma_{\hsigma_x(y)}=\sigma_y$. This means that $\gamma(\lexp{y}{\gamma_{\op}(x)})=\gamma(y)$ where $x,y$ are now considered as the generators of $G(X,r)$ and $\gamma$, respectively $\gamma_\op$, is the gamma function associated to $G(X,r)$, respectively $G(X,r)_\op$. In particular, as $X$ generates the multiplicative group $(G(X,r),\circ)$, it follows that $\gamma(\lexp{y}{\gamma_{\op}(g)})=\gamma(y)$ for all $g\in G(X,r)$. 
    
    For a word $w=x_1^{\epsilon_1}\cdot\ldots \cdot x_n^{\epsilon_n}$ with $x_i\in X$ and $\epsilon_i\in \{-1,1\}$, we will prove that 
    \begin{equation*}
    	\gamma(w)=\gamma(x_n)^{\epsilon_n}\ldots\gamma(x_1)^{\epsilon_1}.
    \end{equation*} As $(G(X,r),\cdot)$ is generated by $X$, this then proves that
    \begin{equation*}
    	\gamma\colon (G(X,r),\cdot)\to \Aut(G(X,r),\cdot)
    \end{equation*} is a group antihomomorphism, and therefore $G(X,r)$ is a bi-skew brace. We will prove this claim by induction on $n$. For $n=1$ and $\epsilon_1=1$ the statement is trivial. To also cover the case where $n=1$ and $\epsilon_1=-1$, we have to prove that $\gamma(x^{-1})=\gamma(x)^{-1}$ for all $x\in X$. For this, we note that there is the equality 
    $\lexpp{\overline{a}}{\gamma_\op(a)}=(a\circ \overline{a})\cdot a^{-1}=a^{-1}$, or equivalently, $\overline{\lexpp{a^{-1}}{\gamma_\op(a)^{-1}}}=a$, so substituting $a$ by $x^{-1}$ we find $\overline{\lexpp{x}{\gamma_\op(x^{-1})^{-1}}}=x^{-1}$, thus
    \begin{equation*}
    	\gamma(x^{-1})=\gamma\left(\overline{\lexpp{x}{\gamma_\op(x^{-1})^{-1}}}\right)=\gamma\left(\lexpp{x}{\gamma_\op(x^{-1})^{-1}}\right)^{-1}=\gamma(x)^{-1}.
    \end{equation*}
    Now assume that the statement holds for words of length $n-1$, and let $w=x_1^{\epsilon_1}\cdot\ldots\cdot x_n^{\epsilon_n}$ be a word of length $n$. If we write $v=x_2^{\epsilon_2}\cdot\ldots\cdot x_n^{\epsilon_n}$, then
    \begin{align*}
        \gamma(w)&=\gamma(v\circ \lexpp{x_1^{\epsilon_1}}{\gamma_{\op}(v)})\\
        &=\gamma(v)\gamma(\lexpp{x_1^{\epsilon_1}}{\gamma_{\op}(v)})\\
        &=\gamma(x_n)^{\epsilon_n}\ldots\gamma(x_2)^{\epsilon_2}\gamma(\lexpp{x_1}{\gamma_{\op}(v)})^{\epsilon_1}\\
        &=\gamma(x_n)^{\epsilon_n}\ldots\gamma(x_2)^{\epsilon_2}\gamma(x_1)^{\epsilon_1}.\qedhere
    \end{align*}
\end{proof}
\begin{corollary}
    Let $(X,r)$ be a solution such that for all $x,y\in X$,
    \begin{equation*}
        \sigma_{\hsigma_x(y)}=\sigma_y.
    \end{equation*} Then $G(X,r)$ is a bi-skew brace. 
\end{corollary}
\begin{proof}
    As $\Inj(X,r)$ is a homomorphic image of $(X,r)$, it follows that $\Inj(X,r)$ still has the property that $\sigma_{\hsigma_x(y)}=\sigma_y$ for all $x,y\in \Inj(X,r)$. Because $G(X,r)$ and $G(\Inj(X,r))$ are isomorphic skew braces, the result follows from Proposition~\ref{theorem: G(X r) is bi-skew}.
\end{proof}
As an application, we find a nice description for involutive solutions such that the skew brace on the structure group is a bi-skew brace.
\begin{proposition}
Let $(X,r)$ be an involutive solution. Then the following statements are equivalent:
\begin{enumerate}
    \item $G(X,r)$ is a bi-skew brace.
    \item $\G(X,r)$ is a trivial brace.
    \item $\Ret(X,r)$ is a trivial solution.
\end{enumerate}
\end{proposition}
\begin{proof}
    As $(X,r)$ is involutive, $\G(X,r)\cong G(X,r)/\ker (\gamma)$, where $\gamma$ is the gamma function of $G(X,r)$. It follows that $G(X,r)^{(3)}=\{1\}$ if and only if $\G(X,r)^{(2)}=\{1\}$. This proves the equivalence of the first two properties. 
    
    Assume that $\G(X,r)$ is a trivial brace, so its associated solution is a trivial solution. Because $\Ret(X,r)$ is equal to the image of the canonical map $\pi\iota\colon(X,r)\to (\G(X,r),r_{\G(X,r)})$, we find that $\Ret(X,r)$ is also trivial.
    
    At last, assume that $\Ret(X,r)$ is a trivial solution. Once again using the fact that $\Ret(X,r)$ embeds into $(\G(X,r),r_{\G(X,r)})$, we know that $\lexp{y}{\gamma(x)}=y$ in $\G(X,r)$ for all $x,y\in \Ret(X,r)\subseteq \G(X,r)$. As $\Ret(X,r)$ generates both the additive and multiplicative group of $\G(X,r)$, it follows that $\lexp{h}{\gamma(g)}=h$ for all $g,h\in \G(X,r)$. This means that $\G(X,r)$ is a trivial skew brace.
\end{proof}

\begin{remark}
    If $(X,r)$ is an involutive solution such that $\Ret(X,r)$ is trivial, then clearly $\Ret(\Ret(X,r))$ is the trivial solution on a singleton. Solutions with this property are said to be of \emph{multipermutation level 2}. Involutive multipermutation solutions of level 2 were classified in \cite{JPZD20} and in particular, \cite[Theorem 7.8]{JPZD20} gives an explicit construction of all involutive solutions with a trivial retract.
\end{remark}

\section{A characterisation of brace blocks, an explicit construction, and some examples}\label{sec: brace blocks}

Recall the definition of a brace block.

\begin{definition}
 Let $A$ be a set. A \emph{brace block}, denoted by $((A,\circ_i)\mid i\in I)$, consists of a family of group operations $(\circ_i\mid i\in I)$ on $A$ such that $(A,\circ_i,\circ_j)$ is a bi-skew brace for all $i,j\in I$. 
\end{definition}

In order to give a characterisation of brace blocks, we begin with a result on transitivity of bi-skew braces. 

\begin{theorem}\label{thm:main}
	Let $(A,\cdot,\circ_1)$ and $(A,\cdot,\circ_2)$ be bi-skew braces with gamma functions $\gamma_1$ and $\gamma_2$, respectively. Then $(A,\circ_1,\circ_2)$ is a bi-skew brace if and only if the following conditions hold: for all $a,b\in A$ and $i,j\in\{1,2\}$ with $i\ne j$,
	\begin{align*}
	   \gamma_i(a)\gamma_j(b)\gamma_i(a)^{-1}=\gamma_j(\lexp{b}{\gamma_i(a)}).
	\end{align*}
\end{theorem}

\begin{proof}
	By symmetry, we can just look at when $(A,\circ_1,\circ_2)$ is a skew brace. As 
\begin{equation*}
    a\circ_2 b=a\cdot \lexp{b}{\gamma_2(a)}=a\circ_1 \lexp{b}{\gamma_1(a)^{-1}\gamma_2(a)},
\end{equation*}
we need to find under which conditions $\gamma(a)=\gamma_1(a)^{-1}\gamma_2(a)$ is a gamma function on $(A,\circ_1)$. 

The first condition to check is whether for all $a\in A$, we have $\gamma(a)\in \Aut(A,\circ_1)$, or equivalently, $\gamma_2(a)\in\Aut(A,\circ_1)$.
Here we have
\begin{gather*}
    \lexpp{b\circ_1 c}{\gamma_2(a)}=\lexpp{b\cdot \lexp{c}{\gamma_1(b)}}{\gamma_2(a)}
    =\lexp{b}{\gamma_2(a)}\cdot\lexp{c}{\gamma_2(a)\gamma_1(b)}\intertext{and}
    \lexp{b}{\gamma_2(a)}\circ_1 \lexp{c}{\gamma_2(a)}=\lexp{b}{\gamma_2(a)}\cdot \lexp{c}{\gamma_1(\lexp{b}{\gamma_2(a)})\gamma_2(a)}.
\end{gather*}
We find that $\gamma(a)\in \Aut(A,\circ_1)$ if and only if for all $a,b\in A$,
\begin{equation}\label{eq:maincondition}
	\gamma_2(a)\gamma_1(b)\gamma_2(a)^{-1}=\gamma_1(\lexp{b}{\gamma_2(a)}).
\end{equation} 
Now suppose that \eqref{eq:maincondition} holds. We claim that it is already enough to deduce that $\gamma\colon (A,\circ_2)\to \Aut(A,\circ_1)$ is a group homomorphism. For all $a,b\in A$, 
\begin{align*}
	\gamma(a\circ_2 b)&=\gamma_1(a\cdot \lexp{b}{\gamma_2(a)})^{-1}\gamma_2(a\circ_2 b)\\
	&=\gamma_1(a)^{-1}\gamma_1(\lexp{b}{\gamma_2(a)})^{-1}\gamma_2(a)\gamma_2(b)\\
	&=\gamma_1(a)^{-1}\gamma_2(a)\gamma_1(b)^{-1}\gamma_2(a)^{-1}\gamma_2(a)\gamma_2(b)\\
	&=\gamma_1(a)^{-1}\gamma_2(a)\gamma_1(b)^{-1}\gamma_2(b)=\gamma(a)\gamma(b).\qedhere
\end{align*}
\end{proof}
\begin{definition}
Let $(A,\cdot)$ be a group. A \textit{brace block on $(A,\cdot)$} is a brace block $((A,\circ_i)\mid i\in I)$ such that $(A,\circ_k)=(A,\cdot)$ for some $k\in I$.
\end{definition}

We deduce the following characterisation for brace blocks on a given group.

 \begin{theorem}\label{thm:mainblock}
	Let $(A,\cdot)$ be a group. Then the following data are equivalent:
\begin{enumerate}
	\item A brace block on $(A,\cdot)$.
	\item A family of maps $(\gamma_i\mid i\in I)$ such that the following conditions hold:
	 \begin{itemize}
	\item $\gamma_i\colon (A,\cdot)\to \Aut(A,\cdot)$ is a group antihomomorphism for all $i\in I$.
	\item There exists $k\in I$ such that $\gamma_k(a)=\id$ for all $a\in A$.
	\item For all $i,j\in I$ and $a,b\in A$,
		\begin{equation*}
			\gamma_i(a)\gamma_j(b)\gamma_i(a)^{-1}=\gamma_j(\lexp{b}{\gamma_i(a)}). 
		\end{equation*} 
\end{itemize}
\end{enumerate}
\end{theorem}
\begin{proof}
	For all $i\in I$, we find that $(A,\cdot,\circ_i)$ is a bi-skew brace because $\gamma_i$ is a group antihomomorphism and the last condition for $i=j$ implies that 
	\begin{equation*}
		\gamma_i(a\cdot \lexp{b}{\gamma_i(a)})=\gamma_i(a)\gamma_i(b).
	\end{equation*}
	Now apply Theorem~\ref{thm:main}.
\end{proof}

\begin{remark}
	A similar condition was found in a particular case in~\cite[Theorem 4.30]{Spa22}, where the problem of finding mutually normalising regular subgroups in the holomorph of a cyclic group of prime order was dealt with. This problem is equivalent to looking for brace blocks; see~\cite[section 7]{CS22a} for more details. 
\end{remark}

We use this characterisation to propose an intermediate construction of brace blocks, which already can provide several examples.

\begin{theorem}\label{theorem: intermediate}
	Let $(A,\cdot)$ be a group, let $M$ be an abelian subgroup of $\Aut(A,\cdot)$, and let $\mathcal{S}$ be the set of group homomorphisms $\gamma\colon A\to M$ such that $\gamma(\psi(a))=\gamma(a)$ for all $a\in A$ and $\psi\in M$. Then
	$((A,\circ_{\gamma})\mid \gamma\in\mathcal{S})$
	is a brace block, where 
	\begin{equation*}
	    a\circ_{\gamma}b=a\cdot \lexp{b}{\gamma(a)}.
	\end{equation*}
	 Moreover, $(A,\circ_{\gamma_1},\circ_{\gamma_2})$ is $\gamma$-homomorphic for all $\gamma_1,\gamma_2\in \mathcal{S}$.
\end{theorem}

\begin{proof}
	The first part is immediate from Theorem~\ref{thm:mainblock}. To conclude the second part, recall that the gamma function $\gamma_{1,2}$ of $(A,\circ_{\gamma_1},\circ_{\gamma_2})$ is given by $\gamma_{1,2}(a)= \gamma_1(a)^{-1}\gamma_2(a)$. As $\gamma_{1,2}(A)\subseteq M$, so in particular it is abelian, we find that $(A,\circ_{\gamma_1},\circ_{\gamma_2})$ is $\gamma$-homomorphic by Lemma~\ref{lemma: two implies three}.
\end{proof}

\begin{example}\label{example: rings}
	Let $R$ be a ring with unity (not necessarily commutative), and for all $x\in R$, define the following map:
\begin{align*}
    \gamma_x\colon R^2&\to \Aut(R^2,+),\quad
    \begin{pmatrix}
r\\
s
\end{pmatrix}\mapsto \begin{pmatrix}
1&0\\ xr & 1
\end{pmatrix}.
\end{align*}
We note the following facts:
\begin{itemize}
	\item For all $x\in R$, \begin{equation*}
		\gamma_x(R^2)\subseteq  M= \left\{\begin{pmatrix}
1&0\\ y & 1
\end{pmatrix}\mid y\in R\right\},
	\end{equation*}
	and $M$ is clearly abelian. 
	\item For all $x\in R$, we have that $\gamma_x$ is a group homomorphism.
	\item For all $x\in R$, $a\in R^2$, and $\psi\in M$, we have $\gamma_x(\psi(a))=\gamma_x(a)$.
\end{itemize}
We conclude that $((R^2,\circ_x)\mid x\in R^2)$ is a brace block, where
\begin{equation*}
	\begin{pmatrix}
r\\
s
\end{pmatrix}
\circ_x
\begin{pmatrix}
r'\\
s'
\end{pmatrix}
=
\begin{pmatrix}
r+r'\\
s+s'+xrr'
\end{pmatrix}.
\end{equation*}
Moreover, all the operations are distinct, because for all $x\in R$,
\begin{equation*}
	\gamma_x\begin{pmatrix}
1\\
0
\end{pmatrix}= \begin{pmatrix}
1&0\\ x & 1
\end{pmatrix}.
\end{equation*}
Assume that $R$ is commutative and that for all $r\in R$ there exists a unique $r'\in R$ such that $2r'=r^2-r$ (by abuse of notation, we say $r'=\frac{r^2-r}{2}$), which is for example the case if $R$ is $\Z$ or an algebra over a field of characteristic different from 2. It is straightforward to verify that we have a group isomorphism
\begin{align*}
	\theta\colon (R^2,+)&\to (R^2,\circ_{\gamma_x}),\quad
	\begin{pmatrix}
r\\
s
\end{pmatrix}\mapsto \begin{pmatrix}
r\\
s+\frac{x(r^2-r)}{2}
\end{pmatrix}.
\end{align*} 

\end{example}
\begin{example}\label{example: Z^2}
    If in the previous example we take $R=\Z$, we find that if $x\neq \pm y$ then $(\Z^2,+,\circ_{\gamma_x})\not\cong (\Z^2,+,\circ_{\gamma_{y}})$. Indeed, $$\begin{pmatrix}
    r\\s
    \end{pmatrix}*\begin{pmatrix}
    r'\\s'
    \end{pmatrix}=\begin{pmatrix}
    0\\xrr'
    \end{pmatrix}$$ and therefore $(\Z^2,+,\circ_{\gamma_x})/(\Z^2,+,\circ_{\gamma_x})^2\cong \Triv(\Z\times C_{|x|})$. This nicely contrasts the case of additive group $\Z$, where only 2 distinct group operations $\circ$ giving a skew brace $(\Z,+,\circ)$ are possible, as we have recalled in section~\ref{sec: classification}.
    
    One can show that the skew braces $(\Z^2,+,\circ_{\gamma_x})$ are isomorphic to the $\lambda$-cyclic skew braces with infinite cyclic image constructed in~\cite[Section 4]{BNY22}. 
\end{example}
\begin{example}
Let us reconsider Example~\ref{example: rings} with $R$ a field of characteristic not 2. As explained above, we obtain a brace block $((R^2,\circ_{\gamma_x})\mid x\in R)$, where $(R^2,+)\cong (R^2,\circ_{\gamma_x})$ for all $x\in R$. We claim that in this case all bi-skew brace of the form $(R^2,\circ_x,\circ_y)$, where $x,y\in R$ and $x\neq y$, are isomorphic. For this, it suffices to check that an isomorphism of skew braces is given by
\begin{align*}
	\theta\colon (R^2,+,\circ_{\gamma_1})&\to (R^2,\circ_{\gamma_x},\circ_{\gamma_y}),\quad
	\begin{pmatrix}
r\\
s
\end{pmatrix}\mapsto
	\begin{pmatrix}
r\\
(y-x)s+\frac{x(r^2-r)}{2}
\end{pmatrix}.
\end{align*}
\end{example}

Let now $G$ and $H$ be groups. We have seen in section~\ref{sec: structural result} that semidirect products of the form $\Triv(G)\ltimes \Triv(H)$, respectively $\opTriv(G)\ltimes \Triv(H)$, are an easy way to construct $\gamma$-homomorphic skew braces, respectively bi-skew braces. It is therefore natural to try to generalise this construction in order to obtain brace blocks. 

For a group homomorphism $\alpha\colon G\to \Aut(H)$, we write $\circ_{\alpha}$ for the group operation on $G\times H$ given by the semidirect product of $G$ and $H$.

\begin{proposition}\label{proposition: brace block semidirect}
Let $G$ and $H$ be groups, let $M$ be an abelian subgroup of $\Aut(H)$, and let $\mathcal{S}$ be the set of group homomorphisms $\alpha\colon G\to M$. Then $((G\times H,\circ_{\alpha})\mid \alpha \in \mathcal{S})$ is a brace block.
\end{proposition}
\begin{proof}
For all $\alpha\in \mathcal{S}$, let $\gamma_\alpha$ be the gamma function associated with $(G\times H,\cdot,\circ_{\alpha})$: 
\begin{align*}
	\gamma_\alpha\colon G\times H\to \Aut(G)\times\Aut(H)\subseteq \Aut(G\times H),\quad (g,h)\mapsto (\id, \alpha_g).
\end{align*}
These images are all contained in the abelian subgroup $\{\id\}\times M$ of $\Aut(G\times H)$. In order to apply Theorem~\ref{theorem: intermediate}, it suffices to check that $\gamma_\alpha(\psi(g,h))=\gamma_\alpha(g,h)$ for all $\alpha\in \mathcal{S}$, $\psi=(\id,\psi')\in \{\id\}\times M$, and $(g,h)\in G\times H$. We find
\begin{align*}
    \gamma_\alpha(\psi(g,h))&=\gamma_\alpha(g,\psi'(h))=(\id,\alpha_g)=\gamma_\alpha(g,h).\qedhere
\end{align*}
\end{proof}
\begin{example}
    For all $n\in \N$, define 
    \begin{align*}
    	\alpha_n\colon \Z\to \GL_2(\Z),\quad 
    x\mapsto\begin{pmatrix}
    1&0\\
    nx&1
    \end{pmatrix},
    \end{align*}
    which is a well-defined group homomorphism. As $$B=\left\{\begin{pmatrix}
    1&0\\
    x&1
    \end{pmatrix}\mid x\in \Z\right\}$$ is clearly an abelian subgroup of $\GL_2(\Z)$, from Proposition~\ref{proposition: brace block semidirect} we obtain a brace block $((\Z\times \Z^2,\circ_n)\mid n\in \N)$. More precisely,
    $$\begin{pmatrix}
    x\\y\\z
    \end{pmatrix}\circ_n\begin{pmatrix}
    x'\\y'\\z'
    \end{pmatrix}=\begin{pmatrix}
    x+x'\\y+y'\\z+z'+nxy'
    \end{pmatrix}$$
    In particular, if $n\neq m$, then $(\Z\times \Z^2,\circ_{\alpha_n})$ and $(\Z\times \Z^2,\circ_{\alpha_m})$ are nonisomorphic, as
    the abelianisation of $(\Z\times \Z^2,\circ_{\alpha_n})$ is isomorphic to $ \Z^2\times C_{n}$. We have thus obtained a brace block with countably many nonisomorphic groups. 
\end{example}

\begin{example}
Let $(A,\cdot)$ be a group, and let $B$ be a subgroup of $A$. Assume that the group of inner automorphisms of $(A,\cdot)$ induced by $B$, which we denote by $H(B)$, is abelian. It is easy to check that this is equivalent to $[B,B]\subseteq Z(A)$. 

A straightforward verification shows that $\gamma(\psi(a))=\gamma(a)$ for all group homomorphisms $\gamma\colon A\to H(B)$, $\psi\in H(B)$, and $a\in A$. This means that if we denote by $\mathcal{S}$ the set of group homomorphisms from $A$ to $H(B)$, we obtain a brace block $((A,\circ_{\gamma})\mid \gamma\in \mathcal{S})$
with 
\begin{equation*}
    a\circ_{\gamma} b=a\cdot\lexp{b}{\gamma(a)}.
\end{equation*}
Note that the groups homomorphisms $A\to H(B)$ correspond precisely to the group homomorphisms $A\to B/(B\cap Z(A))$, because $H(B)\cong B/(B\cap Z(A))$. For example, every group homomorphism $\psi\colon A\to B$ yields a group homomorphism $A\to B/(B\cap Z(A))$, which we can use for our construction. Moreover, we have $\psi[A,A]\subseteq [B,B]\subseteq Z(A)$, so we find precisely the condition described in~\cite[Theorem 1.2]{CS21}. In particular, when $B$ is abelian, we recover~\cite{Koc21, Koc22}. Indeed, all the bi-skew braces found in these works are associated with gamma functions which act by conjugation and have a common abelian codomain; see also Remark~\ref{remark: comparison iterative}. For a concrete application of this construction in Hopf--Galois theory, see \cite[Theorem 4.9]{ST22b-p}.
\end{example}
\begin{example}
    We show now how the main construction of~\cite{CS22a} follows from Theorem~\ref{theorem: intermediate}. 
    Let $(A,\cdot)$ be a group, let $B$ be a subgroup of $(A,\cdot)$ such that $[B,B]$ is contained in $Z(A)$ (so that $H(B)$, defined as before, is abelian), and let $K$ be a subgroup of $B$ contained in $Z(A)$. (Note that we do not require that $B/K$ is abelian.) Define
    \begin{gather*}
         \mathcal{A}=\{\psi\colon A/K\to B/K \text{ group homomorphism}\},\\
          \mathcal{B}=\{\alpha\colon A\times A\to K\mid \text{$\alpha$ is bilinear and $\alpha(A,K)=\alpha(K,A)=\{1\}$}\}.
    \end{gather*}
    For all $\psi\in \mathcal{A}$ and $\alpha\in \mathcal{B}$, define
    \begin{equation*}
        a\circ_{\psi,\alpha} b=a\cdot \psi(a)\cdot b\cdot \psi(a)^{-1}\cdot \alpha(a,b),
    \end{equation*}
    where with a little abuse of notation we write $\psi(a)$ for any element in the coset $\psi(aK)$. In~\cite{CS22a} it is shown that this construction is well-defined, and that moreover $((A,\circ_{\psi,\alpha})\mid (\psi,\alpha)\in \mathcal{A}\times \mathcal{B})$ is a brace block. 
    We can write 
    \begin{equation*}
    	a\circ_{\psi,\alpha} b=a\cdot \lexp{b}{\gamma_1(a)\gamma_2(a)},
    \end{equation*}
    where $\gamma_1(a)$ denotes conjugation by $\psi(a)$ and $\gamma_2(a)\colon b\mapsto b\cdot \alpha(a,b)$. Consider now the following subgroup of the central automorphisms of $(A,\cdot)$:
    \begin{equation*}
        L=\{\delta\in\Aut(A,\cdot)\mid \delta(b)\cdot b^{-1}\in K \text{ and } \delta(k)=k \text{ for all $b\in A$ and $k\in K$}\}.
    \end{equation*}
    The group $L$ is abelian and it centralises the subgroup of inner automorphisms of $\Aut(A,\cdot)$, so that $M=H(B)L$ is abelian. Now define $\mathcal{S}$ as in Theorem~\ref{theorem: intermediate}, with respect to $M$. 
    It is just a matter of computation to show that the group homomorphism $a\mapsto \gamma_1(a)\gamma_2(a)$ is an element of $\mathcal{S}$, so we apply Theorem~\ref{theorem: intermediate} to derive our claim. 
    
    Note that in fact there is no need for the codomain of the maps in $\mathcal{A}$ to be $B/K$. We could just consider group homomorphisms from $A/K$ to $B/(B\cap Z(A))$ and find that the construction still works. This means that we do not require any relation between $K$ and $B$; in this way we find a further generalisation of the original construction.
\end{example}

We now use Theorem~\ref{theorem: intermediate} to obtain an iterative construction of brace blocks.

\begin{corollary}\label{cor:iterate}
	Let $(A,\cdot,\circ)$ be a $\gamma$-homomorphic bi-skew brace, and for all $n\in \Z$ and $a\in A$, let $\gamma_n(a)=\gamma(a^n)=\gamma(a)^n$. Then $((A,\circ_{n})\mid n\in \Z)$
	is a brace block, where
	\begin{equation*}
		a\circ_{n}b=a\cdot \lexp{b}{\gamma_n(a)}.
	\end{equation*} 
\end{corollary}

\begin{proof}
 Apply Theorem~\ref{theorem: intermediate} with $M=\gamma(A)$, which is abelian by Lemma~\ref{lemma: two implies three}.
\end{proof}

\begin{remark}\label{remark: comparison iterative}
	This construction presents some similarities with Koch's construction~\cite{Koc22} (or more precisely, the variation presented in~\cite[Example 5.2]{CS22a}), but the operations we find are different. Indeed, let $(A,\cdot)$ be a group, and let $\psi$ be an abelian endomorphism. Then $(A,\cdot,\circ)$ is a $\gamma$-homomorphic bi-skew brace, where $\gamma(a)$ is conjugation by $\psi(a)$. Both constructions yield a brace block $((A,\circ_n)\mid n\in \Z)$, where
	\begin{equation*}
		a\circ_n b =a\cdot \psi_n(a)\cdot b\cdot\psi_n(a)^{-1}.
	\end{equation*}
	In Koch's case, $\psi_n(a)=\prod_{i=1}^n \psi^i\left(a^{\binom{n}{i}}\right)$, while in our case, $\psi_n(a)=\psi(a^n)$. 
\end{remark}

\begin{remark}
    Corollary~\ref{cor:iterate}, which is a natural application of Theorem~\ref{theorem: intermediate}, also recently appeared in~\cite[Theorem 4.12]{BNY22-p}, where the approach followed is significantly different.
\end{remark}
\begin{example}
Let
\begin{align*}
	\gamma\colon \Z^2&\to \GL_2(\Z),\quad 
\begin{pmatrix}
a\\ b
\end{pmatrix}\mapsto \begin{pmatrix}
1&0\\ a & 1
\end{pmatrix}.
\end{align*}
This clearly is a gamma function which provides a $\gamma$-homomorphic bi-skew brace, and applying Corollary~\ref{cor:iterate} to $\gamma$, we find precisely the $\gamma_n$, $n\in \N$, as in Example~\ref{example: Z^2}. In particular, our iterative construction yields a brace block containing countably many nonisomorphic skew braces. 
\end{example}
It is natural to ask whether a similar construction as Corollary~\ref{cor:iterate} is still possible when we are not necessarily starting from a $\gamma$-homomorphic bi-skew brace. The following proposition shows that this is indeed the case.
\begin{proposition}\label{prop: deformation by endomorphism}
Let $A$ be a skew brace, and let $\psi$ a group endomorphism of $(A,\cdot)$ such that for all $a,b\in A$, the equation
\begin{equation*}
    \psi(\lexp{b}{\gamma(\psi(a))})=\lexpp{\psi(b)}{\gamma(\psi(a))}
\end{equation*} holds. Then $(A,\cdot, \circ_\psi)$ is skew brace, where
\begin{equation*}
    a\circ_\psi b=a\cdot \lexp{b}{\gamma(\psi(a))}.
\end{equation*}
\end{proposition}
\begin{proof}
It suffices to prove that the map $A\mapsto \Aut(A,\cdot)$ given by $a\mapsto \gamma(\psi(a))$ is a gamma function on $(A,\cdot)$. For all $a,b\in A$, we find that
\begin{equation*}
    \gamma(\psi(a\cdot \lexp{b}{\gamma(\psi(a))}))=\gamma(\psi(a)\cdot \lexp{\psi(b)}{\gamma(\psi(a))})=\gamma(\psi(a)\circ \psi(b))=\gamma(\psi(a))\gamma(\psi(b)).\qedhere
\end{equation*}
\end{proof}
The relation with Corollary~\ref{cor:iterate} is especially clear when we look at the following straightforward corollary, which in particular applies to Jacobson radical rings.
\begin{corollary}
    Let $(A,\cdot,\circ)$ be a (two-sided) brace. Then $(A,\cdot,\circ_n)$ is a (two-sided) brace for all $n\in \Z$, where
    \begin{equation*}
        a\circ_n b=a\cdot \lexp{b}{\gamma(a^n)}.
    \end{equation*}
\end{corollary}
\begin{proof}
    The fact that $(A,\cdot,\circ_n)$ is a brace follows by applying Proposition~\ref{prop: deformation by endomorphism}, with $\psi\colon a\to a^n$. A straightforward calculation shows that if $(A,\cdot,\circ)$ is two-sided, then also $(A,\cdot,\circ_n)$ is two-sided.
\end{proof}

\section*{Acknowledgment}
The authors would like to thank the anonymous referee for the careful reading and invaluable suggestions, which contributed to improve the exposition and some arguments for proofs.

\bibliographystyle{amsalpha}
\bibliography{bib.bib}

\newcommand{\etalchar}[1]{$^{#1}$}
\providecommand{\bysame}{\leavevmode\hbox to3em{\hrulefill}\thinspace}
\providecommand{\MR}{\relax\ifhmode\unskip\space\fi MR }
\providecommand{\MRhref}[2]{%
  \href{http://www.ams.org/mathscinet-getitem?mr=#1}{#2}
}
\providecommand{\href}[2]{#2}
\begin{thebibliography}{CGK{\etalchar{+}}21}

\bibitem[AD95]{AD95}
Bernhard Amberg and Oliver Dickenschied, \emph{On the adjoint group of a
  radical ring}, Canad. Math. Bull. \textbf{38} (1995), no.~3, 262--270.
  \MR{1347297}

\bibitem[Bac18]{Bac18}
David Bachiller, \emph{Solutions of the {Y}ang--{B}axter equation associated to
  skew left braces, with applications to racks}, J. Knot Theory Ramifications
  \textbf{27} (2018), no.~8, 1850055, 36. \MR{3835326}

\bibitem[BCJO19]{BCJO19}
D.~Bachiller, F.~Ced\'{o}, E.~Jespers, and J.~Okni\'{n}ski, \emph{Asymmetric
  product of left braces and simplicity; new solutions of the {Y}ang--{B}axter
  equation}, Commun. Contemp. Math. \textbf{21} (2019), no.~8, 1850042, 30.
  \MR{4020748}

\bibitem[BNY22a]{BNY22}
Valeriy~G. Bardakov, Mikhail~V. Neshchadim, and Manoj~K. Yadav, \emph{On
  {$\lambda$}-homomorphic skew braces}, J. Pure Appl. Algebra \textbf{226}
  (2022), no.~6, Paper No. 106961, 37. \MR{4346001}

\bibitem[BNY22b]{BNY22-p}
\bysame, \emph{Symmetric skew braces and brace systems}, arXiv:2204.12247,
  2022.

\bibitem[Byo15]{Byo15}
Nigel~P. Byott, \emph{Solubility criteria for {H}opf-{G}alois structures}, New
  York J. Math. \textbf{21} (2015), 883--903. \MR{3425626}

\bibitem[Car18]{Car18}
A.~Caranti, \emph{Multiple holomorphs of finite {$p$}-groups of class two}, J.
  Algebra \textbf{516} (2018), 352--372. \MR{3863482}

\bibitem[Car20]{Car20}
\bysame, \emph{Bi-skew braces and regular subgroups of the holomorph}, J.
  Algebra \textbf{562} (2020), 647--665. \MR{4130907}

\bibitem[Ced18]{Ced18}
Ferran Ced\'{o}, \emph{Left braces: solutions of the {Y}ang-{B}axter equation},
  Adv. Group Theory Appl. \textbf{5} (2018), 33--90. \MR{3824447}

\bibitem[CGK{\etalchar{+}}21]{CGKKKTU21}
Lindsay~N. Childs, Cornelius Greither, Kevin~P. Keating, Alan Koch, Timothy
  Kohl, Paul~J. Truman, and Robert~G. Underwood, \emph{Hopf algebras and
  {G}alois module theory}, Mathematical Surveys and Monographs, vol. 260,
  American Mathematical Society, Providence, RI, [2021] \copyright 2021.
  \MR{4390798}

\bibitem[Chi00]{Chi00}
Lindsay~N. Childs, \emph{Taming wild extensions: {H}opf algebras and local
  {G}alois module theory}, Mathematical Surveys and Monographs, vol.~80,
  American Mathematical Society, Providence, RI, 2000. \MR{1767499}

\bibitem[Chi19]{Chi19}
\bysame, \emph{Bi-skew braces and {H}opf {G}alois structures}, New York J.
  Math. \textbf{25} (2019), 574--588. \MR{3982254}

\bibitem[CJK{\etalchar{+}}23]{CJKVAV22}
F.~Ced\'{o}, E.~Jespers, {\L}.~Kubat, A.~Van~Antwerpen, and C.~Verwimp,
  \emph{On various types of nilpotency of the structure monoid and group of a
  set-theoretic solution of the {Y}ang--{B}axter equation}, J. Pure Appl.
  Algebra \textbf{227} (2023), no.~2, Paper No. 107194, 38. \MR{4457900}

\bibitem[CS21]{CS21}
A.~Caranti and L.~Stefanello, \emph{From endomorphisms to bi-skew braces,
  regular subgroups, the {Y}ang--{B}axter equation, and {H}opf--{G}alois
  structures}, J. Algebra \textbf{587} (2021), 462--487. \MR{4304796}

\bibitem[CS22]{CS22a}
\bysame, \emph{Brace blocks from bilinear maps and liftings of endomorphisms},
  J. Algebra \textbf{610} (2022), 831--851. \MR{4473766}

\bibitem[CSV19]{CSV19}
Ferran Ced\'{o}, Agata Smoktunowicz, and Leandro Vendramin, \emph{Skew left
  braces of nilpotent type}, Proc. Lond. Math. Soc. (3) \textbf{118} (2019),
  no.~6, 1367--1392. \MR{3957824}

\bibitem[Dri92]{Dri92}
V.~G. Drinfel'd, \emph{On some unsolved problems in quantum group theory},
  Quantum groups ({L}eningrad, 1990), Lecture Notes in Math., vol. 1510,
  Springer, Berlin, 1992, pp.~1--8. \MR{1183474}

\bibitem[ESS99]{ESS99}
Pavel Etingof, Travis Schedler, and Alexandre Soloviev, \emph{Set-theoretical
  solutions to the quantum {Y}ang-{B}axter equation}, Duke Math. J.
  \textbf{100} (1999), no.~2, 169--209. \MR{1722951}

\bibitem[GV17]{GV17}
L.~Guarnieri and L.~Vendramin, \emph{Skew braces and the {Y}ang--{B}axter
  equation}, Math. Comp. \textbf{86} (2017), no.~307, 2519--2534. \MR{3647970}

\bibitem[JPZD20]{JPZD20}
P\v{r}emysl Jedli\v{c}ka, Agata Pilitowska, and Anna Zamojska-Dzienio,
  \emph{The construction of multipermutation solutions of the {Y}ang-{B}axter
  equation of level 2}, J. Combin. Theory Ser. A \textbf{176} (2020), 105295,
  35. \MR{4123749}

\bibitem[Koc21]{Koc21}
Alan Koch, \emph{Abelian maps, bi-skew braces, and opposite pairs of
  {H}opf-{G}alois structures}, Proc. Amer. Math. Soc. Ser. B \textbf{8} (2021),
  189--203. \MR{4273165}

\bibitem[Koc22]{Koc22}
\bysame, \emph{Abelian maps, brace blocks, and solutions to the {Y}ang-{B}axter
  equation}, J. Pure Appl. Algebra \textbf{226} (2022), no.~9, Paper No.
  107047. \MR{4381676}

\bibitem[KSV21]{KSV21}
Alexander Konovalov, Agata Smoktunowicz, and Leandro Vendramin, \emph{On skew
  braces and their ideals}, Exp. Math. \textbf{30} (2021), no.~1, 95--104.
  \MR{4223285}

\bibitem[KT20]{KT20a}
Alan Koch and Paul~J. Truman, \emph{Opposite skew left braces and
  applications}, J. Algebra \textbf{546} (2020), 218--235. \MR{4033084}

\bibitem[LV19]{LV19}
Victoria Lebed and Leandro Vendramin, \emph{On structure groups of
  set-theoretic solutions to the {Y}ang--{B}axter equation}, Proc. Edinb. Math.
  Soc. (2) \textbf{62} (2019), no.~3, 683--717. \MR{3974961}

\bibitem[LYZ00]{LYZ00}
Jiang-Hua Lu, Min Yan, and Yong-Chang Zhu, \emph{On the set-theoretical
  {Y}ang-{B}axter equation}, Duke Math. J. \textbf{104} (2000), no.~1, 1--18.
  \MR{1769723}

\bibitem[Nas19]{Nas19}
Timur Nasybullov, \emph{Connections between properties of the additive and the
  multiplicative groups of a two-sided skew brace}, J. Algebra \textbf{540}
  (2019), 156--167. \MR{4003478}

\bibitem[NV06]{NV06}
Sam Nelson and John Vo, \emph{Matrices and finite biquandles}, Homology
  Homotopy Appl. \textbf{8} (2006), no.~2, 51--73. \MR{2246021}

\bibitem[Rum07a]{Rum07a}
Wolfgang Rump, \emph{Braces, radical rings, and the quantum {Y}ang--{B}axter
  equation}, J. Algebra \textbf{307} (2007), no.~1, 153--170. \MR{2278047}

\bibitem[Rum07b]{Rum07b}
\bysame, \emph{Classification of cyclic braces}, J. Pure Appl. Algebra
  \textbf{209} (2007), no.~3, 671--685. \MR{2298848}

\bibitem[Smo18]{Smo18}
Agata Smoktunowicz, \emph{On {E}ngel groups, nilpotent groups, rings, braces
  and the {Y}ang-{B}axter equation}, Trans. Amer. Math. Soc. \textbf{370}
  (2018), no.~9, 6535--6564. \MR{3814340}

\bibitem[Sol00]{Sol00}
Alexander Soloviev, \emph{Non-unitary set-theoretical solutions to the quantum
  {Y}ang-{B}axter equation}, Math. Res. Lett. \textbf{7} (2000), no.~5-6,
  577--596. \MR{1809284}

\bibitem[Spa22]{Spa22}
Filippo Spaggiari, \emph{The mutually normalizing regular subgroups of the
  holomorph of a cyclic group of prime power order}, Communications in Algebra
  (2022), 1--31.

\bibitem[ST22]{ST22b-p}
L.~Stefanello and S.~Trappeniers, \emph{On the connection between
  {H}opf--{G}alois structures and skew braces}, arXiv:2206.07610, 2022.

\bibitem[SV18]{SV18}
Agata Smoktunowicz and Leandro Vendramin, \emph{On skew braces (with an
  appendix by {N}. {B}yott and {L}. {V}endramin)}, J. Comb. Algebra \textbf{2}
  (2018), no.~1, 47--86. \MR{3763907}

\bibitem[Ven19]{Ven19}
Leandro Vendramin, \emph{Problems on skew left braces}, Adv. Group Theory Appl.
  \textbf{7} (2019), 15--37. \MR{3974481}

\bibitem[Wat68]{Wat68}
J.~F. Watters, \emph{On the adjoint group of a radical ring}, J. London Math.
  Soc. \textbf{43} (1968), 725--729. \MR{229677}

\end{thebibliography}
\end{document}